\documentclass[12pt]{amsart}

\usepackage{latexsym}
\usepackage{amsmath}
\usepackage{amsfonts}
\usepackage{amssymb}
\usepackage{geometry} 
\usepackage{graphicx,color}
\usepackage{epstopdf}
\usepackage{caption}
\usepackage{diagbox}
\usepackage{subfig}
\usepackage{bm}
\usepackage{algorithm}
\usepackage{algorithmic}
\usepackage{mathrsfs}
\graphicspath{{figures/}}

\newtheorem{theorem}{Theorem}[section]
\newtheorem{lemma}[theorem]{Lemma}

\newcommand{\Lt}{L_2(\Omega)}
\def\st{\scriptscriptstyle}

\def\bz{\bm{\zeta}}
\def\asiph{a^{\text{sip}}_h}
\def\aarh{a^{\text{ar}}_h}
\def\Eo{\mathring{E}(\Delta;L_2(\Omega))}
\def\En{E(\Delta;L_2(\Omega))}
\def\bM{\mathbf{M}}
\def\bL{\mathbf{L}}
\def\bA{\mathbf{A}}

\def\bn{\mathbf{n}}
\def\by{\mathbf{y}}
\def\Ho{H^1_0(\Omega)}
\def\O{\Omega}

\def\LT{{L_2(\O)}}
\def\Linf{{L_\infty(\O)}}

\def\cT{\mathcal{T}}
\def\cE{\mathcal{E}}
\def\cL{\mathcal{L}}
\def\cA{\mathcal{A}}

\def\HO{{H^1(\O)}}
\def\vc{\mathring{V}_h^c}

\def\fC{\mathfrak{C}}

\theoremstyle{definition}

\newtheorem{example}[theorem]{Example}

\theoremstyle{remark}
\newtheorem{remark}[theorem]{Remark}

\numberwithin{equation}{section}

\newcommand{\trinorm}[1]{%
  \left|\mkern-2mu\left|\mkern-2mu\left|
   #1
  \right|\mkern-2mu\right|\mkern-2mu\right|
}

\DeclareMathOperator*{\argmin}{argmin}

\begin{document}

\title[DG For An Elliptic Optimal Control Problem]{Discontinuous Galerkin Methods for an Elliptic Optimal Control Problem with a General State Equation and Pointwise State Constraints}

\author{Sijing Liu, Zhiyu Tan and Yi Zhang}
\address{Sijing Liu, Department of Mathematics\\
University Of Connecticut\\
Storrs, CT\\
USA}
\email{sijing.liu@uconn.edu}
\address{Zhiyu Tan, Center for Computation $\&$ Technology\\
Louisiana State University \\
Baton Rouge, LA\\
USA}
\email{ztan@cct.lsu.edu}
\address{Yi Zhang, Department of Mathematics\\
University Of North Carolina Greensboro\\
Greensboro, NC\\
USA}
\email{y\_zhang7@uncg.edu}

\date{\today}

\begin{abstract}
    We investigate discontinuous Galerkin methods for an elliptic optimal control problem with a general state equation and pointwise state constraints on general polygonal domains. We show that discontinuous Galerkin methods for general second-order elliptic boundary value problems can be used to solve the elliptic optimal control problems with pointwise state constraints. We establish concrete error estimates and numerical experiments are shown to support the theoretical results.
\end{abstract}

\maketitle

\section{Introduction}

Let $\Omega$ be a polygonal domain in $\mathbb{R}^2$, $y_d\in L_2(\Omega)$, $\beta$ be a positive constant and $g\in H^4(\O)$. The elliptic optimal control problem is to find
\begin{equation}\label{eq:optcon}
(\bar{y},\bar{u})=\argmin_{(y,u)\in \mathbb{K}_g}\left [ \frac{1}{2}\|y-y_d\|^2_{L_2(\Omega)}+\frac{\beta}{2}\|u\|^2_{L_2(\Omega)}\right],
\end{equation} 
where $(y,u)$ belongs to $\mathbb{K}_g\subset \HO\times \LT$ if and only if 
\begin{equation}\label{eq:stateeq}
\begin{aligned}
a(y,v)&=\int_{\Omega}uv \ dx \quad \forall v\in H^1_0(\Omega),\\
y&=g\quad\text{on}\quad\partial\O,
\end{aligned}
\end{equation}
and the pointwise state constraint
\begin{equation}\label{eq:statecon}
y\le\psi \quad \mbox{a.e. in}\ \ \Omega, 
\end{equation}
where the function
\begin{equation}\label{eq:phiassump}
\psi\in W^{3,p}\ \text{with}\ p>2\ \text{and}\ \psi>g\ \text{on}\ \partial\O.      
 \end{equation} 

The bilinear form $a(\cdot,\cdot)$ is defined as,
\begin{equation}\label{eq:generalsec}
a(y,v)=\int_{\Omega} \nabla y\cdot \nabla v\ dx+\int_{\Omega} (\bz\cdot\nabla y) v\ dx+\int_{\Omega} \gamma yv\ dx,       
\end{equation}
where the vector field $\bz\in [W^{1,\infty}(\Omega)]^2$ and the function $\gamma\in W^1_{\infty}(\Omega)$ is nonnegative. If $\bz\ne0$, then the constraint \eqref{eq:generalsec} is the weak form of a general second order PDE with an advective/convective term. We assume 
\begin{equation}\label{eq:advassump}
    \gamma-\frac12\nabla\cdot\bz\ge\gamma_0>0
\end{equation}
such that the problem \eqref{eq:stateeq} is well-posed.

Here and throughout the paper we will follow the standard notation for differential operators,
 function spaces and norms that can be found for example in \cite{Ciarlet,BS}.

We define the subspace $\Eo$ of $H^1_0(\Omega)$ as
\begin{equation}
    \Eo=\{y\in H^1_0(\Omega):\mathcal{L} y\in \LT\},
\end{equation}
where $\mathcal{L}y=-\Delta y+\bz\cdot\nabla y+\gamma y$.
We also denote
\begin{equation}
    \En=\{y\in H^1(\Omega):\mathcal{L} y\in \LT\}.
\end{equation}

Due to the elliptic regularity \cite{Dauge,nazarov2011elliptic}, $\Eo$ is a subspace of $H^{1+\alpha}(\Omega)\cap H^2_{loc}(\O)\cap H^1_0(\O)$ for some $\alpha\in(\frac12,1]$, where $\alpha=1$ if $\Omega$ is convex, and
\begin{equation}\label{eq:globalreg}
\|z\|_{H^{1+\alpha}(\Omega)}\le C_{\Omega}\|\cL z\|_{L_2(\Omega)}\quad\forall z\in \Eo.
\end{equation}
It follows from \eqref{eq:globalreg} and the Sobolev inequality \cite[Theorem 4.12]{adams2003sobolev} that $g+\Eo\subset C(\bar{\Omega})$. Therefore we can reformulate \eqref{eq:optcon}-\eqref{eq:statecon} as the following,
\begin{equation}\label{eq:optcon1}
\bar{y}=\argmin_{y\in K_g}\left [ \frac{1}{2}\|y-y_d\|^2_{L_2(\Omega)}+\frac{\beta}{2}\|\cL y\|^2_{L_2(\Omega)}\right],
\end{equation} 
where
\begin{equation}\label{eq:kgdef}
K_g=\{y\in g+\Eo: y\le\psi\ \ \mbox{in}\ \ \Omega\}.
\end{equation}

 Optimal control problems with pointwise state constraints are more difficult to analyze due to the low regularity of the Lagrange multiplier. In \cite{brenner2017new,casas2014new}, the authors proved that the Lagrange multiplier $\mu$ is a nonnegative Borel measure and $\mu\in H^{-1}(\O)$ at the same time. By using this regularity result, the pointwise state constraints can then be handled. 

In the case $\bz=\mathbf{0}$ and $\gamma=0$, the distributed optimal control problem with pointwise state constraints \eqref{eq:optcon}-\eqref{eq:statecon} is investigated in \cite{POne,casas2014new,meyer2008error,liu2009new,rosch2012posteriori, cherednichenko2009error, hintermuller2009moreau, hintermuller2014length} using $P_1$ finite element methods. Several extensions \cite{brenner2020p1,brenner2021p1} to the new approach in \cite{POne} have been established. Other methods are proposed for \eqref{eq:optcon}-\eqref{eq:statecon}, for example, $C^0$ interior penalty methods \cite{brenner2018c,BSZ}, Morley finite element methods \cite{brenner2018morley}, and virtual element methods \cite{brenner2021ac}. Fast solvers for \eqref{eq:optcon}-\eqref{eq:statecon} are also studied in \cite{brenner2023multigrid,brenner2020additive}. We refer to \cite{brenner2020finite} for a more detailed survey about finite element methods for \eqref{eq:optcon}-\eqref{eq:statecon}.
 Overall, as pointed out in \cite{brenner2020finite}, optimal control problems with pointwise state constraints can be analyzed using known finite element methods for fourth order boundary value problems. This crucial observation opens doors to many possible numerical methods for optimal control problems with pointwise state constraints. 

Recently, discontinuous Galerkin (DG) methods for \eqref{eq:optcon}-\eqref{eq:statecon} with $\bz=\mathbf{0}$ and $\gamma=0$ were proposed and analyzed in \cite{BGS_DG} where a new interior maximum estimate \cite{DG_interior} was utilized. For the general case where $\bz\ne\mathbf{0}$ and $\gamma\ne0$, a continuous $P_1$ finite element method was proposed and analyzed in \cite{brenner2021p1}. The goal of this paper is to extend the results in \cite{BGS_DG,brenner2021p1} to an optimal control problem with a general state equation \eqref{eq:optcon}-\eqref{eq:statecon}.  The reason for using a DG method is to enable a fast solution of the discrete problem by the primal-dual active algorithm that converges superlinearly (see Remark \ref{remark:dgmass}). Since the bilinear form \eqref{eq:statecon} is nonsymmetric, we must employ an adjoint consistent \cite{arnold2002unified} method \eqref{eq:ardef} in order to obtain estimates involving $R_h$ and $E_h$ which are key ingredients in the convergence analysis.

 The rest of the paper is organized as follows. In Section \ref{sec:contprob}, we gather some known regularity results for the continuous problem \eqref{eq:optcon}-\eqref{eq:statecon}. These results are useful in the convergence analysis. In Section \ref{sec:disprob}, we propose mixed discontinuous Galerkin methods to solve the problem \eqref{eq:optcon}-\eqref{eq:statecon} and establish some important properties of the discrete problem. Three crucial operators $E_h$, $R_h$ and $\fC_h$ are defined and analyzed in Section \ref{sec:preesti}.  
 Concrete error estimates are established in Section \ref{sec:convanalysis}. Numerical results are provided in Section \ref{sec:numerics} and we end with some concluding remarks in Section \ref{sec:conremark}. A description of the primal-dual active set algorithm is given in Appendix \ref{apdix:pdas} and the proofs of Lemma \ref{lemma:ah} and Lemma \ref{lemma:rhestimates} are provided in Appendix \ref{apdix:pfrh}.

Throughout this paper, we use $C$
 (with or without subscripts) to denote a generic positive
 constant that is independent of any mesh
 parameters.
  Also to avoid the proliferation of constants, we use the
   notation $A\lesssim B$ (or $A\gtrsim B$) to
  represent $A\leq \text{(constant)}B$. The notation $A\approx B$ is equivalent to
  $A\lesssim B$ and $B\lesssim A$.

\section{The Continuous Problem}\label{sec:contprob}

Let $\bar{z}=\bar{y}-g$. Then the problem \eqref{eq:optcon1}-\eqref{eq:kgdef} is equivalent to find
\begin{equation}\label{eq:optcon2}
\bar{z}=\argmin_{z\in \widetilde{K}}\left [ \frac{1}{2}\|z-(y_d-g)\|^2_{L_2(\Omega)}+\frac{\beta}{2}\|\cL (z+g)\|^2_{L_2(\Omega)}\right],
\end{equation} 
where
\begin{equation}
\widetilde{K}=\{z\in \Eo: z\le\psi-g\ \ \mbox{in}\ \ \Omega\}.
\end{equation}

By the classical theory of variation of calculus, it is well-known that there exists a unique solution $\bar{z}$ to \eqref{eq:optcon2}. Consequently, the problem \eqref{eq:optcon1} has a unique solution and $\bar{y}$ can be characterized by
\begin{equation}\label{eq:vi}
(\bar{y}-y_d ,y-\bar{y})_{\Lt}+\beta(\cL\bar{y},\cL (y-\bar{y}))_{\Lt}\ge 0\quad \forall\ y\in K_g.
\end{equation}

\vspace{0.3cm}
\noindent{\bf Interior regularity of $\bar{y}$}
\vspace{0.3cm}

By the interior regularity results for fourth order variational inequalities \cite{frehse1971,frehse1973regularity}, we have $\bar{z}\in H^3_{loc}(\Omega)\cap W^{2,\infty}_{loc}(\Omega)$. Since $g\in H^4(\O)$, which is a subspace of $W^{2,\infty}(\O)$ by the Sobolev inequality, we conclude
\begin{equation}\label{eq:interreg}
    \bar{y}\in H^3_{loc}(\Omega)\cap W^{2,\infty}_{loc}(\Omega).
\end{equation}
A proof can be found in \cite{brenner2021p1}.

\vspace{0.3cm}
\noindent{\bf Lagrange multiplier $\mu$}
\vspace{0.3cm}

Taking $y=-\phi+\bar{y}\in K_g$ in \eqref{eq:vi} where $\phi$ is a nonnegative function in $C^{\infty}_0(\Omega)$. Thus we have 
 \begin{equation}\label{eq:functional}
 \int_{\Omega}  \Big[ (\bar{y}-y_d)\phi+\beta(\cL\bar{y})(\cL \phi)\Big] dx \le 0.
 \end{equation}
It follows from {\cite[Section 13, Theorem 25]{royden1988real}} or {\cite[Theorem 2.14]{rudin2006real}} that 
\begin{equation}\label{eq:rrt}
\int_{\Omega}  \Big[ (\bar{y}-y_d)z+\beta(\cL\bar{y})(\cL z)\Big] dx = \int_{\Omega} z\ d\mu\quad \forall z\in \Eo,
\end{equation}
where
\begin{equation}\label{eq:mumeasure}
     \mu\ \text{is a non-positive regular Borel measure}. 
 \end{equation} 
Furthermore, it can be proved \cite{brenner2017new,casas2014new} that 
\begin{equation}\label{eq:mureg}
    \mu\in H^{-1}(\O),
\end{equation} 
and we have the following complimentarity condition

\begin{equation}\label{eq:comp}
\int_{\Omega} (\psi-\bar{y})\ \!d\mu=0.
\end{equation}

\vspace{0.3cm}
\noindent{\bf Regularity of $\bar{u}$ and global regularity of $\bar{y}$}
\vspace{0.3cm}

It can be proved \cite{brenner2018c,POne} that $\bar{u}\in \Ho$. According to \eqref{eq:globalreg}, we have 
\begin{equation}\label{eq:regu}
\bar{y}\in H^{1+\alpha}(\Omega),
\end{equation}
 where $\alpha\in(\frac12,1]$ and 
 \begin{equation}
\|\bar{y}\|_{H^{1+\alpha}(\Omega)}\le C_{\Omega}\|\cL\bar{y}\|_{L_2(\Omega)}.
\end{equation}

\section{The Discrete Problem}\label{sec:disprob}

Let $\mathcal{T}_h$ be a shape regular simplicial triangulation of $\Omega$.
The diameter of $T\in\mathcal{T}_h$ is denoted by $h_T$ and $h=\max_{T\in\mathcal{T}_h}h_T$ is the mesh diameter. 
Let $\mathcal{E}_h=\mathcal{E}^b_h\cup\mathcal{E}^i_h$ where $\cE^i_h$ (resp. $\cE^b_h$) represents the set of interior edges (resp. boundary edges).

We further decompose the boundary edges $\cE^b_h$ into the inflow part $\cE^{b,-}_h$ and the outflow part $\cE^{b,+}_h$ which are defined as follows,
\begin{align}
    \cE^{b,-}_h&=\{e\in\cE^b_h: e\subset\{x\in\partial\O: \bz(x)\cdot\mathbf{n}(x)<0\}\},\\
    \cE^{b,+}_h&=\cE^b_h\setminus\cE^{b,-}_h.
\end{align}

For an edge $e\in \mathcal{E}^i_h$, let $h_e$ be the length of $e$. For each edge we associate a fixed unit normal $\mathbf{n}$. We denote by $T^+$ the element for which $\mathbf{n}$ is the outward normal, and $T^-$ the element for which $-\mathbf{n}$ is the outward normal. We define the discontinuous finite element space $V_h$ as 
\begin{equation}
    V_h=\{v\in\LT:v|_T\in\mathbb{P}_1(T)\quad\forall\ T\in\mathcal{T}_h\}.
\end{equation}
For $v\in V_h$ on an edge $e$, we define
\begin{equation}
    v^+=v|_{T^+}\quad\text{and}\quad v^-=v|_{T^-}.
\end{equation}
We define the jump and average for $v\in V_h$ on an edge $e$ as follows,
\begin{equation}
    [v]=v^+-v^-,\quad \{v\}=\frac{v^++v^-}{2}.
\end{equation}
For $e\in\mathcal{E}_h^b$ with $e\in\partial T$, we let
\begin{equation}
    [v]=\{v\}=v|_T.
\end{equation}
We also denote 
\begin{equation}
    (w,v)_e:=\int_e wv\ \!ds\quad\text{and}\quad(w,v)_T:=\int_T wv\ \!dx.
\end{equation}

\subsection{Mixed discontinuous Galerkin methods}
We define the piecewise $H^s$ space with $s>\frac32$ as
\begin{equation}\label{eq:broh1}
    H^s(\O;\mathcal{T}_h)=\{v\in \LT: v|_T\in H^s(T) \quad\forall\ T\in\mathcal{T}_h\}.
\end{equation}
Define  $\cL_{h,g}: H^s(\O;\mathcal{T}_h)\rightarrow V_h$ as 
\begin{equation}\label{eq:ddef}
\begin{aligned}
    (\cL_{h,g}w, v)_\LT&=a_h(w, v)+\sum_{e\in\mathcal{E}_h^b}(g,\mathbf{n}\cdot\nabla v-\frac{\sigma}{h_e}v)_e\\
    &\hspace{0.5cm}+\sum_{e\in\cE_h^{b,-}}(\bn\cdot\bz g,v)_e\quad \forall v\in V_h,
    \end{aligned}
\end{equation}
where
\begin{equation}
    a_h(w,v)=a_h^{\text{sip}}(w,v)+a^{\text{ar}}_h(w,v)\quad\forall w,v\in V_h.
\end{equation}
Here 
\begin{equation}\label{eq:dgbilinear}
\begin{aligned}
    a^{\text{sip}}_h(w,v)=&\sum_{T\in\mathcal{T}_h}(\nabla w, \nabla v)_T-\sum_{e\in\mathcal{E}_h}(\{\mathbf{n}\cdot\nabla w\},[v])_e
    -\sum_{e\in\mathcal{E}_h}(\{\mathbf{n}\cdot\nabla v\},[w])_e\\
    &+\sigma\sum_{e\in\mathcal{E}_h} h_e^{-1}([w],[v])_e
\end{aligned}
\end{equation}
is the bilinear form of the symmetric interior penalty (SIP) method with sufficiently large penalty parameter $\sigma$ and
\begin{equation}\label{eq:ardef}
\begin{aligned}
    a^{\text{ar}}_h(w,v)=\sum_{T\in\mathcal{T}_h}(\bz\cdot\nabla w+\gamma w, v)_T
     -\sum_{e\in\cE^i_h\cup\cE^{b,-}_h}(\bn\cdot\bz[w],\{v\})_e.
\end{aligned}
\end{equation}
is the unstabilized DG scheme for advection and reaction terms (cf. \cite{brezzi2004discontinuous} and \cite[Section 2.2]{di2011mathematical}).

\begin{remark}
   We do not consider convection-dominated case in this paper. Therefore the bilinear form $\aarh(\cdot,\cdot)$ does not contain any stabilization terms. However, if one considers convection-dominated case, the following well-known \cite{leykekhman2012local,ayuso2009discontinuous} upwind scheme can be utilized,
    \begin{equation}
         a^{\text{ar}}_h(w,v)=\sum_{T\in\mathcal{T}_h}(\bz\cdot\nabla w+\gamma w, v)_T-\sum_{e\in\cE^i_h}(\bn\cdot\bz[w],v^{\text{up}})_e-\sum_{e\in\cE^{b,-}_h}(\bn\cdot\bz\ \!w,v)_e,
     \end{equation} 
     where the upwind value $v^{\text{up}}$ of a function on an interior edge $e\in\cE_h^i$ is defined as
\begin{equation}\label{eq:upwvalue}
    v^{\text{up}}=\left\{
    \begin{aligned}
    v^+\quad\text{if}\quad\bz\cdot\mathbf{n}\ge0,\\
    v^-\quad\text{if}\quad\bz\cdot\mathbf{n}<0.
    \end{aligned}
    \right.
\end{equation}
\end{remark}

Then the discrete problem for \eqref{eq:optcon1} is to find
\begin{equation}\label{eq:discretecon}
\bar{y}_h=\argmin_{y_h\in K_h}\left [ \frac{1}{2}(y_h-y_d,y_h-y_d)_{\Lt}+\frac{\beta}{2}(\cL_{h,g} y_h,\cL_{h,g} y_h)_{\Lt}\right],
\end{equation}
where
\begin{equation}\label{eq:khdef}
K_h=\{y\in V_h: y_T(p)\le \psi(p)\ \ \mbox{for all}\ \ p\in\mathcal{V}_T\ \text{and all}\ T\in\mathcal{T}_h\}.
\end{equation}
Here $\mathcal{V}_T$ is the set of vertices of $T\in\cT_h$. Note that we impose the Dirichlet boundary condition $y=g$ weakly through $\cL_{h,g}$.

\begin{remark}\label{remark:dgmass}
    The discrete problem \eqref{eq:discretecon}-\eqref{eq:khdef} can be solved by a primal-dual
active set method (see Appendix \ref{apdix:pdas}). Let $\bM_h$ denote the mass matrix represent the bilinear form $(\cdot,\cdot)_\LT$ with respect to the natural discontinuous nodal basis in $V_h$. Note that the computation of $\cL_{h,g}$ involves $\bM_h^{-1}$. In contrast to \cite{POne,brenner2020p1,brenner2021p1}, the matrix $\bM_h$ is block diagonal hence $\bM_h^{-1}$ can be obtained easily.
\end{remark}

\subsection{Properties of $a_h(\cdot,\cdot)$}

Let $D$ be a subdomain of $\O$ and $\cT_h(D)$ be a collection of all the elements with a nonempty intersection with $D$. Define a mesh-dependent norm on $V_h$,
\begin{equation}\label{eq:henergynorm}
    \trinorm{v}_{h,D}^2=\sum_{T\in\mathcal{T}_h(D)}\left[\|\nabla v\|^2_{L_2(T)}+\sum_{e\in\partial T}\frac{\sigma}{h_e}\|[v]\|^2_{L_2(e)}+\sum_{e\in\partial T}\frac{h_e}{\sigma}\|\{\bn\cdot\nabla v\}\|^2_{L_2(e)}\right].
\end{equation}
We use $\trinorm{v}_{h}$ to denote the norm $\trinorm{v}_{h,\O}$ if there is no ambiguity. The following lemma establishes the continuity and coercivity of $a_h(\cdot,\cdot)$, which is standard for the discontinuous Galerkin methods \cite{arnold2002unified,riviere2008discontinuous,BS}. The proof is provided in Appendix \ref{apdix:pfrh}.
\begin{lemma}\label{lemma:ah}
    We have
        \begin{alignat}{3}
            a_h(w,v)&\le C\trinorm{w}_{h}\trinorm{v}_{h}\quad&&\forall w,v\in \En+V_h,\label{eq:ahcont}\\
            a_h(v,v)&\ge C\trinorm{v}^2_{h}\quad&&\forall v\in V_h,\label{eq:ahcoer}
        \end{alignat}
    for large enough $\sigma$.
\end{lemma}

\subsection{Properties of $\cL_{h,g}$}
By the definition of $\cL_{h,g}$ and integration by parts, we have for any $w\in g+\Eo$ 
\begin{equation}
\begin{aligned}
    (\cL_{h,g}w, v)_\LT&=a_h(w, v)_\LT+\sum_{e\in\mathcal{E}_h^b}(g,\mathbf{n}\cdot\nabla v-\frac{\sigma}{h_e}v)_e+\sum_{e\in\cE_h^{b,-}}(\bn\cdot\bz g,v)_e\\
    &=(\cL w, v)_\LT\quad \forall v\in V_h.
    \end{aligned}
\end{equation}
which gives the consistency of $\cL_{h,g}$. Moreover, we have
\begin{equation}\label{eq:l2proj}
    \cL_{h,g}w=Q_h\cL w\quad\forall w\in g+\Eo, 
\end{equation}
where $Q_h$ is the orthogonal projection from $\LT$ onto $V_h$.

\subsection{Discrete variational inequalities}
Define $\cL_h: H^s(\O;\mathcal{T}_h)\rightarrow V_h$ as
\begin{equation}\label{eq:ddef0}
    (\cL_hw, v)_\LT=a_h(w, v)_\LT\quad \forall v\in V_h.
\end{equation}
Note that we have the following,
\begin{equation}\label{eq:deltahrela}
    \cL_{h,g}w_1-\cL_{h,g}w_2=\cL_h(w_1-w_2)\quad\forall w_1,w_2\in H^s(\O;\mathcal{T}_h).
\end{equation}
In particular, we have
\begin{equation}\label{eq:l2projzero}
    \cL_hw=\cL_{h,0}w=Q_h\cL w\quad \forall w\in\Eo.
\end{equation}

 Since $K_h$ is a nonempty closed convex set and the objective function in
\eqref{eq:discretecon} is strongly convex, the discrete problem \eqref{eq:discretecon}-\eqref{eq:khdef} has a unique solution $\bar{y}_h\in V_h$, which can be characterized by
\begin{equation}\label{eq:vidis}
(\bar{y}_h-y_d ,y_h-\bar{y}_h)_{\Lt}+\beta(\cL_{h,g}\bar{y}_h,\cL_h (y_h-\bar{y}_h))_{\Lt}\ge 0\quad \forall\ y_h\in K_h.
\end{equation}

\section{Preliminary Estimates}\label{sec:preesti}

In this section, we establish some preliminary estimates for the convergence analysis. We consider both quasi-uniform meshes and graded meshes \cite{brannick2008uniform,babuvska1970finite,apel1996graded} around reentrant corners.
\subsection{Graded Meshes}

For a nonconvex domain with reentrant corners, it is well-known that the solution to the state equation \eqref{eq:stateeq} does not belong to $H^2(\O)$ in general (see \eqref{eq:regu}). To overcome this lack of regularity, we can use a triangulation $\cT_h$ with the following properties. Let $\omega_1, \omega_2, \ldots, \omega_L$ be the interior angles at the corners $c_1, c_2, \ldots, c_L$ of the bounded polygonal domain $\O$ and $c_T$ be the center of $T\in\cT_h$. There exists constants $C_1$ and $C_2$ such that
\begin{equation}\label{eq:gradeddef}
    C_1h_T\le\Phi_\mu(T)h\le C_2h_T,\quad\forall\ T\in\cT_h,
\end{equation}
where $\Phi_\mu(T)=\Pi_{l=1}^L|c_l-c_T|^{1-\mu_l}$. Here the grading parameters $\mu_1, \mu_2, \ldots, \mu_L$ are chosen as,
\begin{equation}
    \begin{aligned}
        \mu_l=1,&\quad \omega_l<\pi,\\
        \frac12<\mu_l<\frac{\pi}{\omega_l},&\quad \omega_l>\pi.
    \end{aligned}
\end{equation}
The construction of graded meshes that satisfy \eqref{eq:gradeddef} can be found in \cite{brannick2008uniform,babuvska1970finite,apel1996graded}. 

\subsection{Preliminary inequalities}
The following standard inequalities \cite{arnold2002unified,riviere2008discontinuous,BGS_DG} are needed. Assume $D$ is a subdomain of $\O$ such that $D\Subset\O$, i.e., the closure of $D$ is a compact set of $\O$. Note that $\cT_h$ is quasi-uniform around $D$ for graded meshes. Then a standard inverse estimate implies
        \begin{align}
            \trinorm{v_h}_{h,D}&\le Ch^{-1}\|v_h\|_{L_2(\O)}\quad\forall v_h\in V_h\label{eq:invlocal}.
        \end{align}
We also have the following discrete Sobolev inequality \cite{BGS_DG}
\begin{equation}\label{eq:discretesobo}
    \|y_h\|_\Linf\le C(1+|\ln h|)^\frac12\trinorm{y_h}_h\quad\forall y_h\in V_h.
\end{equation}
For $T\in\mathcal{T}_h$ and $v\in H^{1+s}(T)$ where $s\in(\frac12,1]$, the following trace inequalities with scaling is standard (cf. \cite[Lemma 7.2]{ern2017finite} and \cite[Proposition 3.1]{ciarlet2013analysis}),
    \begin{align}
         \|v\|_{L_2(\partial T)}&\le C(h_T^{-\frac12}\|v\|_{L_2(T)}+h_T^{s-\frac12}|v|_{H^s(T)}),\label{eq:traceinq}\\
         \|\nabla v\|_{L_2(\partial T)}&\le C(h_T^{-\frac12}\|\nabla v\|_{L_2(T)}+h_T^{s-\frac12}|\nabla v|_{H^s(T)}).\label{eq:traceinq1}
    \end{align}
The following discrete Poinca{\'r}e inequality for DG functions \cite{brenner2003poincare,ayuso2009discontinuous,chen2004pointwise} is valid for all $v\in V_h$,
\begin{equation}\label{eq:dgpoin}
    \|v\|^2_\LT\le C\left(\sum_{T\in\cT_h}\|\nabla v\|^2_{L_2(T)}+\sum_{e\in\cE_h}\frac{1}{h_e}\|[v]\|^2_{L_2(e)}\right).
\end{equation}

\subsection{Interpolation operator $I_h$}

Let $V_h^c$ be the conforming $P_1$ finite element space associated with $\cT_h$. 
We use the usual continuous nodal interpolant $I_h: \En\rightarrow V^c_h$ (which belongs to $V_h$) \cite{arnold2002unified,riviere2008discontinuous} such that the following holds.

\begin{lemma}\label{lem:interpo}
We have 
\begin{equation}\label{eq:interpolation}
  \|z-I_hz\|_{\Lt}+h\trinorm{z-I_hz}_h\lesssim h^{1+\tau}\|\cL z\|_{\Lt}\quad\forall z\in \En.
\end{equation}
For quasi-uniform or graded meshes, we also have
\begin{equation}\label{eq:interlinf}
  \|z-I_hz\|_{L^{\infty}(\Omega)}\le Ch^\tau\|\cL z\|_{\Lt}\quad\forall z\in \En.
\end{equation}

Here $\tau$ is defined by
\begin{equation}\label{eq:tau}
\tau=\left\{
\begin{array}{cl}
\alpha&\mbox{if}\ \mathcal{T}_h\ \mbox{is quasi-uniform,}\\
\\
1&\mbox{if}\ \mathcal{T}_h\ \mbox{is graded around the reentrant corners,}
\end{array}
\right.
\end{equation}
where $\alpha\in (\frac{1}{2},1]$ is the index of elliptic regularity in \eqref{eq:globalreg}.
\end{lemma}

The following lemma is useful in the convergence analysis.

\begin{lemma}\label{lemma:deltahih}
    Let $\phi$ be a $C^\infty$ function with compact support in $\O$. We have
\begin{equation}
  \|\cL_h(I_h\phi)\|_{\Lt}\lesssim \|\phi\|_{H^2(\O)}.
\end{equation}
\end{lemma}

\begin{proof}
    Let $D\Subset\O$ be an open neighborhood of the support of $\phi$ and $v_h\in V_h$, we have, by \eqref{eq:ddef0}, \eqref{eq:invlocal}, \eqref{eq:ahcont} and \eqref{eq:interpolation},
\begin{align*}
  (\cL_h(\phi-I_h\phi),v_h)_{\Lt}&=a_h(\phi-I_h\phi,v_h)\\
  &\lesssim \trinorm{\phi-I_h\phi}_h\trinorm{v_h}_{h,D}\\
  &\lesssim h|\phi|_{H^2(\Omega)}\trinorm{v_h}_{h,D}\\
  &\lesssim|\phi|_{H^2(\Omega)}\|v_h\|_{\Lt}, 
\end{align*}
and therefore,
\begin{equation}\label{eq:lhesti}
  \|\cL_h(\phi-I_h\phi)\|_{\Lt}\lesssim |\phi|_{H^2(\Omega)}.
\end{equation}

It follows from \eqref{eq:l2projzero} and \eqref{eq:lhesti} that
\begin{equation}
\begin{aligned}
  \|\cL_h(I_h\phi)\|_{\Lt}&\le\|\cL_h(\phi-I_h\phi)\|_{\Lt}+\|\cL_h\phi\|_{\Lt}\\
  &\lesssim |\phi|_{H^2(\O)}+\|Q_h\cL\phi\|_{\Lt}\\
  &\lesssim |\phi|_{H^2(\O)}+\|\cL\phi\|_{\Lt}\\
  &\lesssim \|\phi\|_{H^2(\O)}.
\end{aligned}
\end{equation}
\end{proof}

\subsection{The Ritz Projection Operator $R_h$}\label{subsec:rh}

We define an operator $R_h: \En\rightarrow V_h$ as the following,
\begin{equation}\label{eq:rh}
    a_h(R_hw,v)=a_h(w,v) \quad\forall v\in V_h.
\end{equation}
It follows from \eqref{eq:ddef}, \eqref{eq:l2proj} and \eqref{eq:rh} that
\begin{equation}\label{eq:rhlhg}
    \cL_{h,g}(R_hw)=\cL_{h,g}w=Q_h(\cL w)\quad\forall w\in g+\Eo.
\end{equation}
Moreover, we have, by \eqref{eq:ddef0},
\begin{equation}\label{eq:rhlh}
    \cL_hR_hw=\cL_hw\quad \forall w\in\En.
\end{equation}
\begin{lemma}\label{lemma:rhestimates}
We have the following error estimates for $R_h$, 
\begin{align}
     \trinorm{w-R_hw}_h&\le Ch^\tau\|\cL w\|_\LT\quad\forall w\in\En,\label{eq:rhh}\\
     \|w-R_hw\|_\LT&\le Ch^{2\tau}\|\cL w\|_\LT\quad\forall w\in\En.\label{eq:rhl2}
 \end{align}    
\end{lemma}
\begin{proof}
    The proof can be found in Appendix \ref{apdix:pfrh}.
\end{proof}

 We have the following interior error estimate \cite{BGS_DG,DG_interior,chen2004pointwise} for $D\Subset\O$,
 \begin{equation}\label{eq:rhinterior}
     \|\bar{y}-R_h\bar{y}\|_{L_\infty(D)}\le C((1+|\ln h|)h^2+h^{2\tau}).
 \end{equation}
 It follows from \eqref{eq:discretesobo} and \eqref{eq:rhh} that
 \begin{equation}\label{eq:rhglobal}
     \|\bar{y}-R_h\bar{y}\|_{\Linf}\le C((1+|\ln h|)h^\tau.
 \end{equation}

\subsection{The Smoothing Operator $E_h$}

The operator $E_h: V_h\rightarrow \Eo$ is defined by 
\begin{equation}\label{eq:eh}
    \cL E_hv_h=\cL_hv_h\quad\forall v_h\in V_h.
\end{equation}
In particular, for all $w\in V_h$,
\begin{equation}\label{eq:eh1}
\begin{aligned}
    a_h(E_hv_h, w)=(\cL E_hv_h, w)_\LT=(\cL_hv_h,w)_\LT=a_h(v_h, w).
    \end{aligned}
\end{equation}
Note that \eqref{eq:eh1} and \eqref{eq:rh} imply
\begin{equation}\label{eq:rheh}
    v_h=R_h(E_hv_h)\quad\forall v_h\in V_h.
\end{equation}

\begin{lemma}\label{lem:ehlemma}
    We have the following estimates for all $v_h\in V_h$ 
    \begin{align}
        \trinorm{E_hv_h-v_h}_{h}&\lesssim h^\tau\|\cL_hv_h\|_{\Lt},\label{eq:ehesti2}\\
        \|E_hv_h-v_h\|_{\Lt}&\lesssim h^{2\tau}\|\cL_hv_h\|_{\Lt},\label{eq:ehesti3}\\
         \trinorm{E_hv_h-v_h}_{h,G(\mathscr{A})}&\lesssim h\|\cL_hv_h\|_{\Lt},\label{eq:ehesti4}
    \end{align}
    where $G(\mathscr{A})$ is an open neighborhood of the active set $\mathscr{A}=\{x\in\O: \bar{y}(x)=\psi(x)\}$ such that the closure of $G(\mathscr{A})$ is a compact subset of $\O$.
\end{lemma}

\begin{proof}
    The estimates follow from \eqref{eq:rheh} and \eqref{eq:rhh}-\eqref{eq:rhinterior}. Details can be found in \cite{BGS_DG}.
\end{proof}

\subsection{The Connection Operator $\fC_h$} We need a connection operator $\fC_h:V_h\rightarrow \mathring{V}_h^c=V_h\cap \Ho$ defined as follows,
\begin{equation}\label{eq:ch}
    (\fC_hy_h)(p)=\frac{1}{|\cT_h(p)|}\sum_{T\in\cT_h(p)}(y_h|_T)(p)
\end{equation}
where $p$ is any node of the $P_1$ finite element space interior to $\O$ and $\cT_h(p)$ is the set of triangles in $\cT_h$ that share the node $p$. The operator $\fC_h$ has the following properties by \eqref{eq:khdef} and \eqref{eq:ch},
\begin{align}
    \fC_hv_h&=v_h\quad\forall v_h\in\vc,\label{eq:ch1}\\
    \fC_hy_h&\in K_h\quad\forall y_h\in K_h.\label{eq:ch2}
\end{align}
For any subdomain $D$ of $\O$, we have,
\begin{equation}\label{eq:chestimate}
\begin{aligned}
    h^{-2}\|y_h-\fC_hy_h\|^2_{L_2(D)}+\sum_{T\in\cT_h}|y_h-\fC_hy_h|^2_{H^1(T)}&\le C\sum_{T\in\cT^\ast_h(D)}\sum_{e\in\partial T}|e|^{-1}\|[y_h]\|_{L_2(e)}^2,
\end{aligned}
\end{equation}
where $T\in\cT_h$ belongs to $\cT^\ast_h(D)$ if and only if $S_T\cap D=\emptyset$. Here $S_T$ (the star of $T$) is the union of all the triangles in $\cT_h$ that share a common vertex with $T$. We also have (cf. \cite{BGS_DG})
\begin{equation}\label{eq:chinf}
    \|\fC_hy_h\|_{L_\infty(T)}\le C\|y_h\|_{L_\infty(S_T)}\quad\forall T\in \cT_h.
\end{equation}
It follows from \eqref{eq:ehesti4} and \eqref{eq:henergynorm} that, 
\begin{equation}\label{eq:chestimate1}
    \sum_{T\in\cT^\ast_h(D)}\sum_{e\in\partial T}|e|^{-1}\|[y_h]\|_{L_2(e)}^2\le Ch^2\|\cL_hy_h\|^2_\LT\quad\quad\forall y_h\in V_h.
\end{equation}

\begin{remark}
    Due to the facts that $V_h\not\subset\Ho$ and $\vc\subset\Ho$, the operator $\fC_h$ is utilized to connect the discrete problem with the continuous problem.
\end{remark}

\section{Convergence Analysis}\label{sec:convanalysis}

In this section, we derive error estimates of the discontinuous Galerkin methods \eqref{eq:discretecon}. We define a mesh-dependent norm
\begin{equation}\label{eq:hnorm}
    \|v\|_h^2=(v,v)_\LT+\beta(\cL_hv,\cL_hv)_\LT. 
\end{equation}

\subsection{An abstract error estimate}
Let $\bar{y}_h\in K_h$ be the solution of \eqref{eq:discretecon}. Given any $y_h\in K_h$, we have the following by \eqref{eq:vidis} and \eqref{eq:deltahrela},
\begin{equation}\label{eq:generalerror}
\begin{aligned}
      \|y_h-\bar{y}_h\|_h^2&=(y_h-\bar{y}_h,y_h-\bar{y}_h)_\LT+\beta(\cL_h(y_h-\bar{y}_h),\cL_h(y_h-\bar{y}_h))_\LT\\
      &=(y_h-\bar{y},y_h-\bar{y}_h)_\LT+\beta(\cL_{h}(y_h-\bar{y}),\cL_h(y_h-\bar{y}_h))_\LT\\
      &\hspace{0.3cm}+(\bar{y}-y_d,y_h-\bar{y}_h)_\LT+\beta(\cL_{h,g}\bar{y},\cL_h(y_h-\bar{y}_h))_\LT\\
      &\hspace{0.3cm}-(\bar{y}_h-y_d,y_h-\bar{y}_h)_\LT-\beta(\cL_{h,g}\bar{y}_h,\cL_h(y_h-\bar{y}_h))_\LT\\
      &\le\|y_h-\bar{y}\|_h\|y_h-\bar{y}_h\|_h+(\bar{y}-y_d,y_h-\bar{y_h})_\LT\\
      &\hspace{0.3cm}+\beta(\cL_{h,g}\bar{y},\cL_h(y_h-\bar{y}_h))_\LT.
\end{aligned}
\end{equation}
Notice that $E_h(y_h-\bar{y}_h)\in \Eo$, we then obtain, by \eqref{eq:rrt}, \eqref{eq:eh} and \eqref{eq:l2proj},
\begin{equation}\label{eq:Ehmeasure}
\begin{aligned}
    &(\bar{y}-y_d,y_h-\bar{y}_h)_\LT+\beta(\cL_{h,g}\bar{y},\cL_h(y_h-\bar{y}_h))_\LT\\
    &\hspace{0.3cm}=(\bar{y}-y_d,(y_h-\bar{y}_h)-E_h(y_h-\bar{y}_h))_\LT\\
    &\hspace{0.6cm}+(\bar{y}-y_d,E_h(y_h-\bar{y}_h))_\LT+\beta(\cL \bar{y},\cL E_h(y_h-\bar{y}_h))_\LT\\
    &\hspace{0.3cm}=(\bar{y}-y_d,(y_h-\bar{y}_h)-E_h(y_h-\bar{y}_h))_\LT+\int_\O E_h(y_h-\bar{y}_h)\ \!d\mu.
    \end{aligned}
\end{equation}

It follows from \eqref{eq:ehesti3} that
\begin{equation}\label{eq:generalerror1}
    (\bar{y}-y_d,(y_h-\bar{y}_h)-E_h(y_h-\bar{y}_h))_\LT\lesssim h^{2\tau}\|\cL_h(y_h-\bar{y}_h)\|_\LT\lesssim h^{2\tau}\|y_h-\bar{y}_h\|_h.
\end{equation}

For the last term in \eqref{eq:Ehmeasure}, we have the following lemma.
\begin{lemma}\label{lemma:interioreh}
    We have, for any $y_h, \bar{y}_h\in K_h$,
    \begin{equation}
        \int_\O E_h(y_h-\bar{y}_h)\ \!d\mu\le C\left(h\|\cL_h(y_h-\bar{y}_h)\|_\LT+h^2+\|I_h(\bar{y}-\fC_hy_h)\|_{L_\infty(\mathscr{A})}\right),
    \end{equation}
    where $C$ is independent of $h$ and $\mathscr{A}$ is the active set.
\end{lemma}

\begin{proof}
    We give a sketch of the proof. The details can be found in \cite[Theorem 4.1]{BGS_DG}. We have, by \eqref{eq:khdef} and \eqref{eq:mumeasure},
    \begin{equation}\label{eq:ehmuesti}
        \begin{aligned}
            \int_\O E_h(y_h-\bar{y}_h)\ \!d\mu&=\int_\O [E_h(y_h-\bar{y}_h)-\fC_h(y_h-\bar{y}_h)]\ \!d\mu+\int_\O I_h(\psi-\fC_h\bar{y}_h)\ \!d\mu\\
            &\hspace{0.2cm}+\int_\O[I_h\fC_h(\bar{y}_h-y_h)-\fC_h(\bar{y}_h-y_h)]\ \!d\mu+\int_\O I_h(\bar{y}-\psi)\ \!d\mu\\
            &\hspace{0.2cm}+\int_\O I_h(\fC_hy_h-\bar{y})\ \!d\mu\\
            &=T_1+T_2+T_3+T_4+T_5.
        \end{aligned}
    \end{equation}
    We can bound $T_2$, $T_4$ and $T_5$ as follows. It follows from \eqref{eq:mumeasure}, \eqref{eq:khdef} and \eqref{eq:ch2} that
    \begin{equation}\label{eq:t2}
        T_2=\int_\O I_h(\psi-\fC_h\bar{y}_h)\ \!d\mu\le 0.
    \end{equation}
    We also have
    \begin{equation}\label{eq:t4}
    \begin{aligned}
        T_4=\int_\O I_h(\bar{y}-\psi)\ \!d\mu&=\int_\O [I_h(\bar{y}-\psi)-(\bar{y}-\psi)]\ \!d\mu\\
        &\le\|I_h(\bar{y}-\psi)-(\bar{y}-\psi)\|_{L_\infty(\mathscr{A})}|\mu(\O)|\le Ch^2,
    \end{aligned}
    \end{equation}
    and
    \begin{equation}\label{eq:t5}
        T_5=\int_\O I_h(\fC_hy_h-\bar{y})\ \!d\mu\le\|I_h(\fC_hy_h-\bar{y})\|_{L_\infty(\mathscr{A})}|\mu(\O)|
    \end{equation}
    by \eqref{eq:phiassump}, \eqref{eq:mumeasure}, \eqref{eq:comp} and \eqref{eq:interreg}.
    For $T_1$, it follows from \eqref{eq:mureg}, \eqref{eq:ehesti4}, \eqref{eq:chestimate} and \eqref{eq:chestimate1} that
    \begin{equation}\label{eq:t1}
        T_1=\int_\O [E_h(y_h-\bar{y}_h)-\fC_h(y_h-\bar{y}_h)]\ \!d\mu\le Ch\|\cL_h(y_h-\bar{y}_h)\|_\LT.
    \end{equation}
    Finally, we obtain
    \begin{equation}\label{eq:t3}
        T_3=\int_\O[I_h\fC_h(\bar{y}_h-y_h)-\fC_h(\bar{y}_h-y_h)]\ \!d\mu\le Ch\|\cL_h(y_h-\bar{y}_h)\|_\LT
    \end{equation}
    by \eqref{eq:mureg}, \eqref{eq:chestimate}, \eqref{eq:chestimate1} and \eqref{eq:interpolation}.
\end{proof}
\begin{remark}
    The key ingredients in the proof of Lemma \ref{lemma:interioreh} is the regularity result \eqref{eq:mureg} and the connection operator $\fC_h$ which is previously used in the analysis of nonconforming methods (cf. \cite{brenner2003poincare}).
\end{remark}

From \eqref{eq:generalerror},\eqref{eq:generalerror1} and Lemma \ref{lemma:interioreh}, we have the following abstract error estimate,
\begin{equation}\label{eq:absesti}
    \|\bar{y}-\bar{y}_h\|_h\le C\left(h+\inf_{y_h\in K_h}\left[\|y_h-\bar{y}\|_h+\|I_h(\bar{y}-\fC_hy_h)\|^\frac12_{L_\infty(\mathscr{A})}\right]\right).
\end{equation}

\subsection{Concrete error estimates}

\begin{lemma}\label{lemma:concreteesti}
    For sufficiently small $h$, there exists $y^\ast_h\in K_h$ such that
    \begin{equation}\label{eq:concreteesti}
        \|y^\ast_h-\bar{y}\|_h+\|I_h(\bar{y}-\fC_hy^\ast_h)\|^\frac12_{L_\infty(\mathscr{A})}\le C((1+|\ln h|)^\frac12h+h^\tau),
    \end{equation}
    where $\tau$ is defined in \eqref{eq:tau}.
\end{lemma}
\begin{proof}
    We give a sketch of the proof. The details can be found in \cite[Lemma 5.1]{BGS_DG}. The crucial task here is to construct a suitable $y_h^\ast\in K_h$ to bound the infimum in \eqref{eq:absesti}. Let $G(\mathscr{A})$ be an open neighborhood of the active set $\mathscr{A}$ and $\epsilon_h=\|R_h\bar{y}-\bar{y}\|_{L_\infty(G(\mathscr{A}))}$. We then define $y^\ast_h\in V_h$ as the following,
    \begin{equation}
        y^\ast_h=R_h\bar{y}-\epsilon_hI_h\phi,
    \end{equation}
    where $\phi\in C^\infty$ is a nonnegative function with compact support in $\O$ such that $\phi=1$ on $G(\mathscr{A})$. We can show that $y^\ast_h\in K_h$. Hence, it follows from Lemma \ref{lemma:deltahih}, \eqref{eq:rhl2} and \eqref{eq:rhlh} that
    \begin{align}
        \|\bar{y}-y^\ast_h\|_\LT&\le C\epsilon_h,\\
        \|\cL_h(\bar{y}-y^\ast_h)\|_\LT&\le C\epsilon_h.
    \end{align}
Therefore, we obtain, by \eqref{eq:hnorm} and \eqref{eq:rhinterior},
\begin{equation}
    \|\bar{y}-y^\ast_h\|_h\le C((1+|\ln h|)h^2+h^{2\tau}).
\end{equation}
At last, it follows from \eqref{eq:chinf}, \eqref{eq:ch1}, \eqref{eq:rhinterior} and \eqref{eq:interpolation} that,
\begin{equation}
    \|I_h(\bar{y}-\fC_hy^\ast_h)\|^\frac12_{L_\infty(\mathscr{A})}\le C((1+|\ln h|)^\frac12h+h^\tau).
\end{equation}
\end{proof}

The following theorem (cf. \cite{BGS_DG}) provide concrete error estimates for the discontinuous Galerkin method \eqref{eq:discretecon}.
\begin{theorem}\label{theorem:main}
    Let $\bar{y}_h\in K_h$ be the solution of \eqref{eq:discretecon} and $\bar{u}_h=\cL_{h,g}\bar{y}_h$. We have
    \begin{equation*}
    \begin{aligned}
        \|\bar{u}-\bar{u}_h\|_\LT+\|\bar{y}-\bar{y}_h\|_\LT+\trinorm{\bar{y}-\bar{y}_h}_{h}+\|\bar{y}-\bar{y}_h\|_\Linf\le C((1+|\ln h|)^\frac12h+h^\tau),
    \end{aligned}
    \end{equation*}
    where $\tau$ is defined in \eqref{eq:tau}.
\end{theorem}

\section{Numerical Results}\label{sec:numerics}

In this section we report the numerical results from two examples. The discrete problem \eqref{eq:discretecon} is solved by a primal-dual active set algorithm \cite{bergounioux2002primal,pdas2002}. For all examples, we take the penalty parameter $\sigma=6$ and the regularization parameter $\beta=1$. For simplicity, we also take $\gamma=1$. We utilized the MATLAB\textbackslash C$++$ toolbox FELICITY \cite{walker2018felicity} in our computation.

\begin{example}[Square Domain \cite{brenner2021p1}]\label{ex:diskactiveset}
    For this example, we take $\O=[-4,4]^2$, $g=0$, $\bz=[1,\ 0]^t$ and $\psi=|x|^2-1$ in \eqref{eq:discretecon}. We consider the function $y_d$ defined as follows, 
    \begin{equation*}
y_d=\left\{\begin{array}{cc}
\cL^t\cL\bar{y}+\bar{y} & |x|>1\\
\cL^t\cL\bar{y}+\bar{y}+2& |x|\le1
\end{array}\right..
\end{equation*}
The function $\bar{y}$ is given by
\begin{equation*}
\bar{y}=\left\{
\begin{array}{cc}
|x|^2-1&|x|\le 1\\
v(|x|)+(1-\phi(|x|))w(x)&1\le|x|\le3\\
w(x)&|x|\ge3
\end{array}\right.,
\end{equation*}
where
\begin{align*}
&v(|x|)=(|x|^2-1)(1-\frac{|x|-1}{2})^4+\frac{1}{4}(|x|-1)^2(|x|-3)^4,\\
&\phi(|x|)=(1+4\frac{|x|-1}{2}+10(\frac{|x|-1}{2})^2+20(\frac{|x|-1}{2})^3)(1-\frac{|x|-1}{2})^4,\\
&w(x)=2\sin(\frac{\pi}{8}(x_1+4))^3\sin(\frac{\pi}{8}(x_2+4))^3.
\end{align*}

By construction, the function $\bar{y}$ is the exact solution and $\bar{y}\le\psi$. The active set is the unit disk $\{x: |x|<1\}$.
The convergence rates on uniform meshes are reported in Table \ref{table:diskp1fem}. As we can see, the convergence rate is around $O(h^2)$ (in average) for the $L_2$  and $L_\infty$ error of the state and $O(h^\frac32)$ for the $L_2$ error of the control. These are better than the estimates in Theorem \ref{theorem:main} and consistent with the fact $\bar{y}\in H^{\frac72-\varepsilon}(\O)$. We also observe $O(h)$ convergence of the state in $\|\cdot\|_{h}$ which is consistent with Theorem \ref{theorem:main}. See Figure \ref{fig:diskactive} for the optimal state, control and active set at level 6.
\end{example}

{\scriptsize
\begin{center}
\captionof{table}{Convergence results for Example \ref{ex:diskactiveset} on uniform meshes with $\bz=[1\ 0]^t$}
\vspace{0.3cm}
\begin{tabular}{|c|c|c|c|c|c|c|c|c|}\hline\label{table:diskp1fem}
$k$&$\|\bar{y}-y_h\|_{L_2(\Omega)}$&Order&$\|\bar{y}-y_h\|_{h}$&Order&$\|\bar{u}-u_h\|_{L_2(\Omega)}$&Order&$\|\bar{y}-y_h\|_{L^\infty(\O)}$&Order\\[1.5ex]
\hline
$1$&3.00e+01&-&3.14e+01&-&6.23e+01&-&8.45e+00&-\\[0.8ex]
\hline
$2$&1.02e+01&1.55&9.33e+00&1.75&2.66e+01&1.23&4.68e+00&0.85\\[0.8ex]
\hline
$3$&2.70e+00&1.92&3.34e+00&1.48&9.33e+00&1.51&9.50e-01&2.30\\[0.8ex]
\hline
$4$&7.45e-01&1.86&1.19e+00&1.48&3.54e+00&1.40&2.80e-01&1.76\\[0.8ex]
\hline
$5$&1.40e-01&2.41&4.30e-01&1.47&1.39e+00&1.35&5.03e-02&2.48\\[0.8ex]
\hline
$6$&6.74e-02&1.05&1.95e-01&1.14&5.59e-01&1.31&2.11e-02&1.25\\[0.8ex]
\hline
$7$&3.76e-02&0.84&9.58e-02&1.02&2.08e-01&1.43&1.28e-02&0.72\\[0.8ex]
\hline
$8$&7.18e-03&2.39&4.31e-02&1.15&6.75e-02&1.62&2.56e-03&2.32\\[0.8ex]
\hline
\end{tabular}
\end{center}}

\begin{figure}[ht]
    \centering
    \subfloat{\includegraphics[height=1.8in]{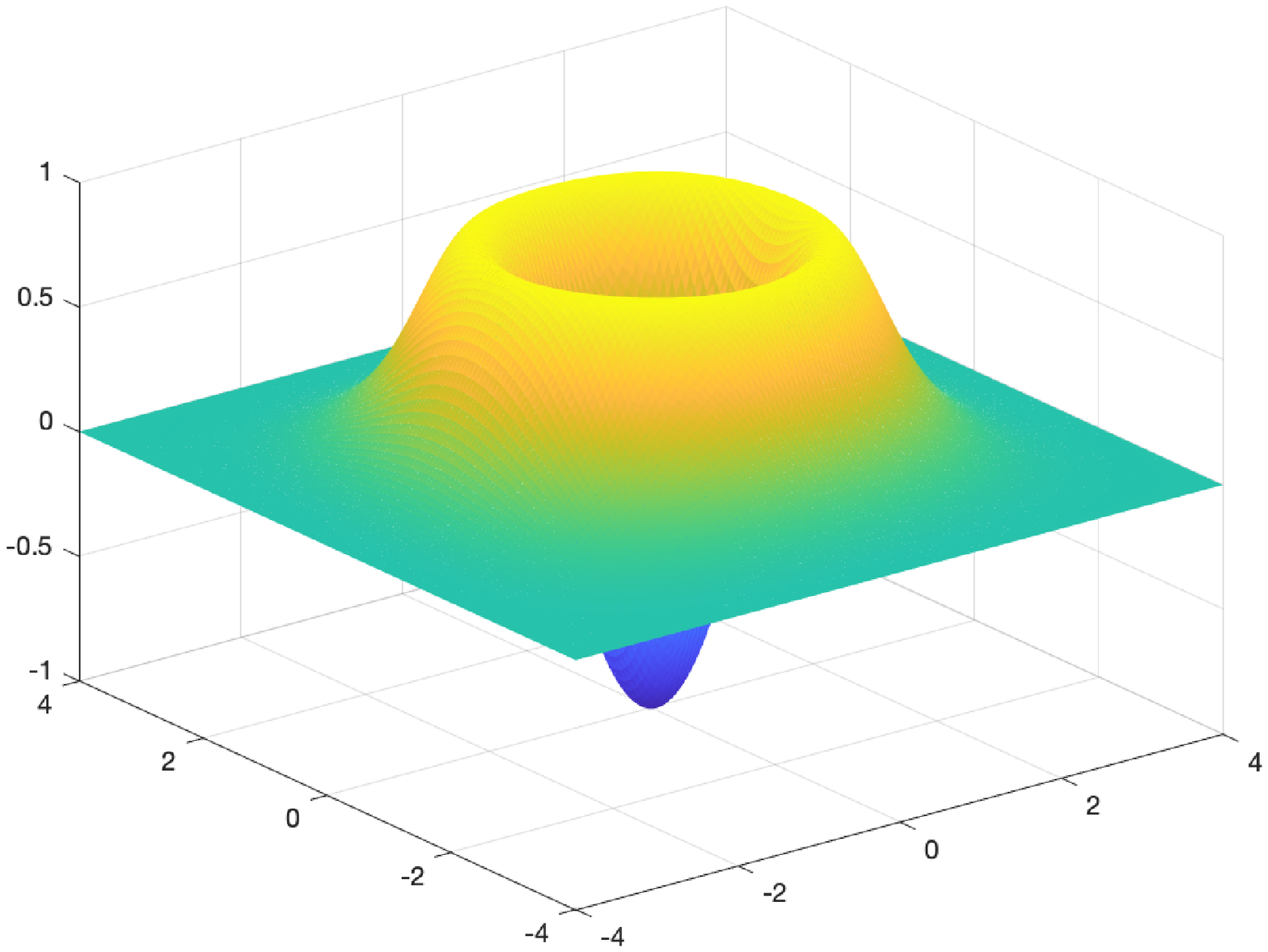}} 
    \subfloat{\includegraphics[height=1.8in]{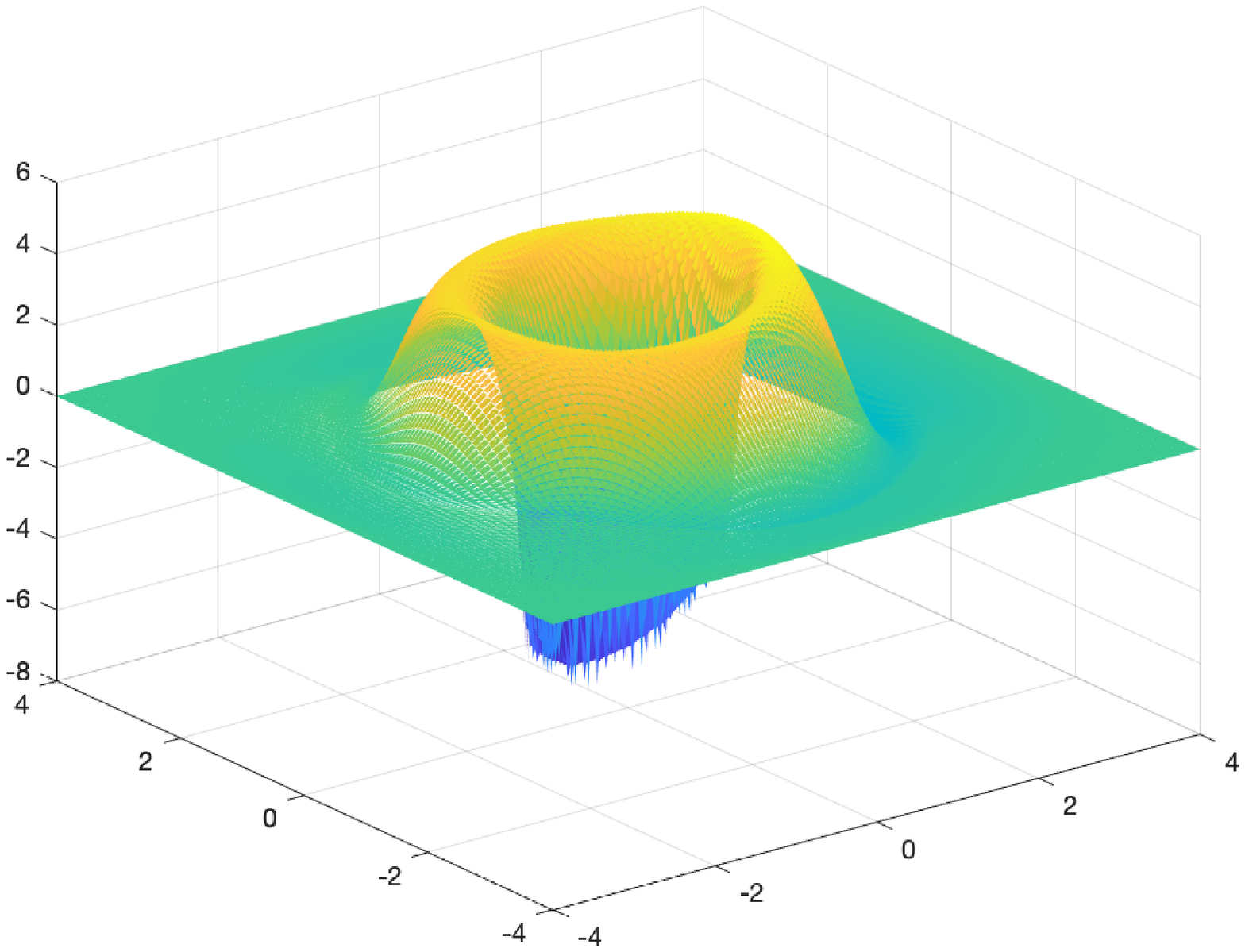}}
    \vfill
    \subfloat{\includegraphics[height=1.8in]{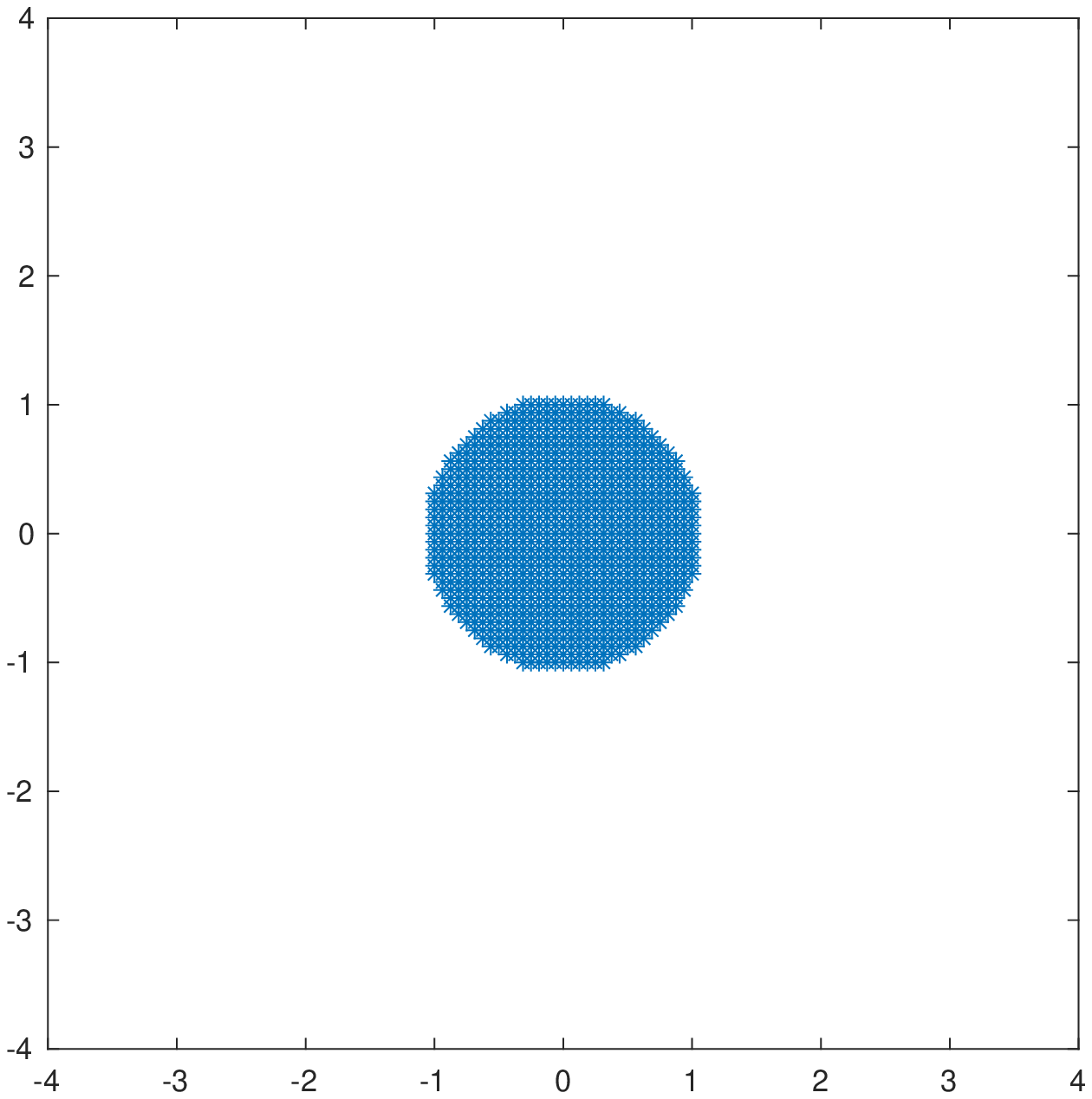}}
    \caption{State, control and active set at level 6 for Example \ref{ex:diskactiveset}}\label{fig:diskactive}
\end{figure}

\begin{figure}[ht]
    \centering
    \subfloat{\includegraphics[height=1.5in]{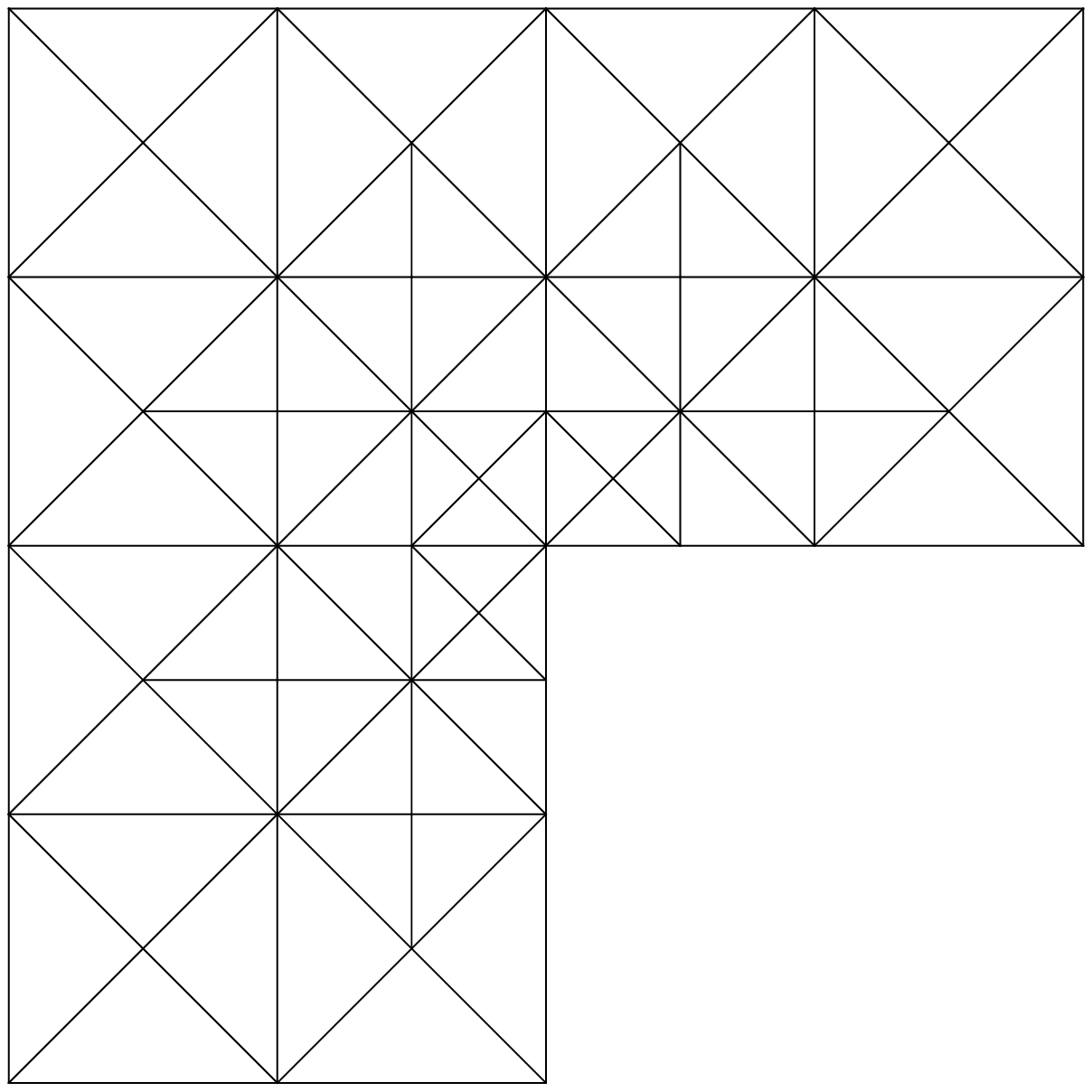}} 
    \subfloat{\includegraphics[height=1.5in]{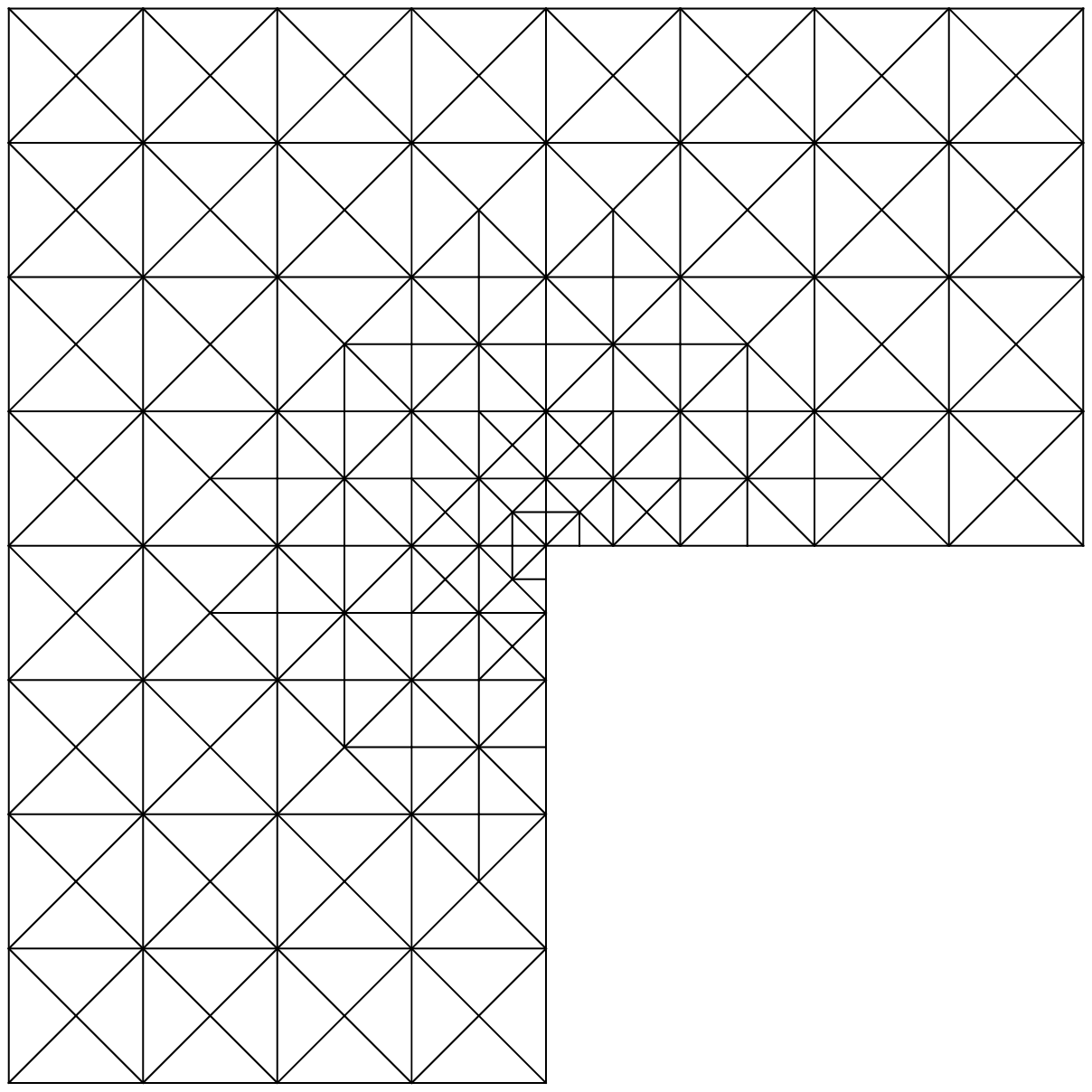}}
    \subfloat{\includegraphics[height=1.5in]{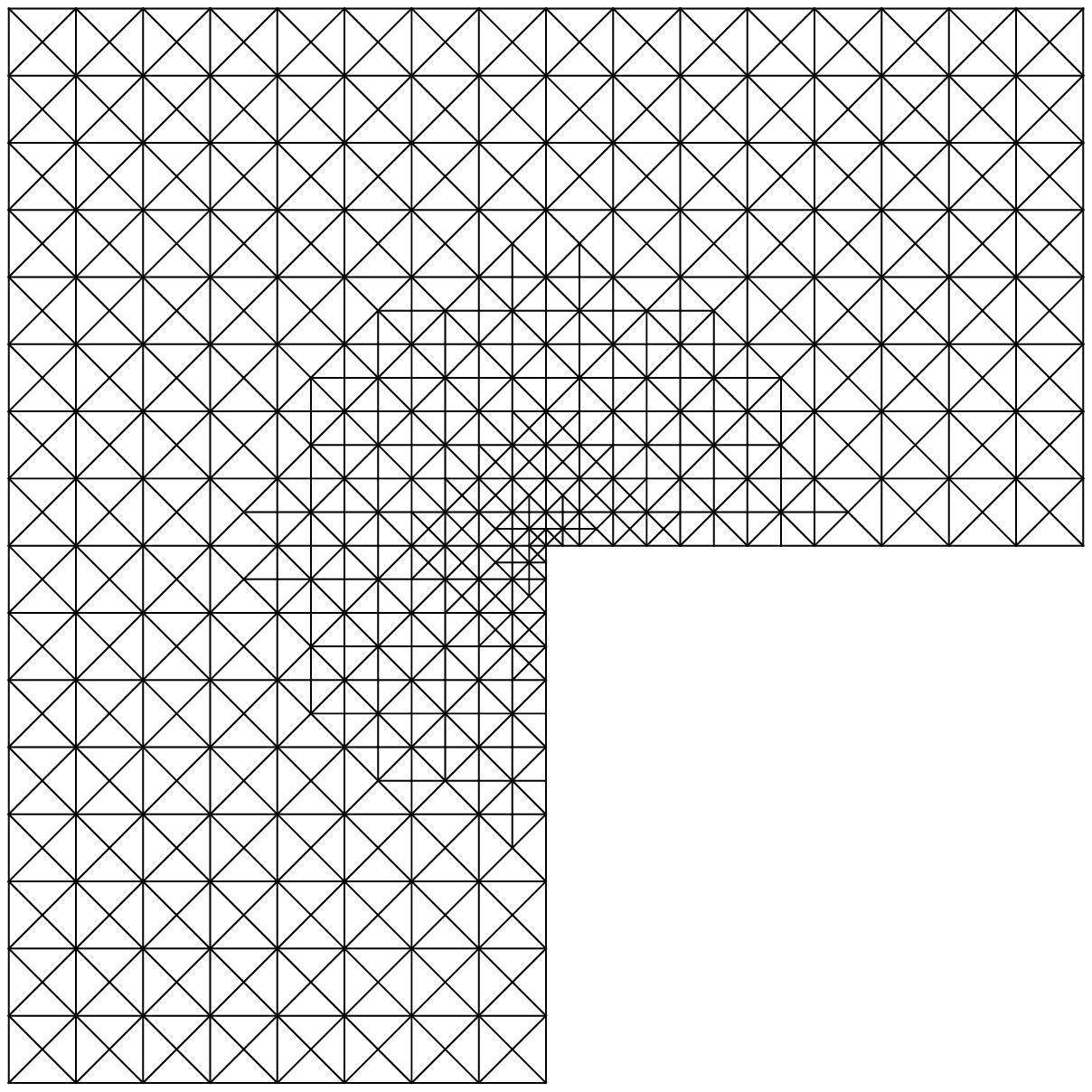}}
    \caption{Graded meshes on L-shaped domain}\label{fig:lshapedgraded}
\end{figure}

\begin{example}[L-shaped Domain \cite{brenner2021p1}]\label{ex:lshaped}
    For this example, we take $\O=[-8,8]^2\setminus[0,8]\times[-8,0]$, $g=10$ and $\bz=[2,\ 1]^t$ in \eqref{eq:discretecon}. This example is a modification of Example \ref{ex:diskactiveset}. The functions $y_d$ and $\psi$ are shifted using the point $x_\ast=(-4,4)$. After that, a singular function $10\psi_s$ is added to the functions $y_d$ and $\psi$. By construction, the exact solution is $\tilde{y}=\bar{y}(x-x_\ast)+10\psi_s$ where $\bar{y}$ is defined in Example \ref{ex:diskactiveset}. Here the function $\psi_s$ is defined by
    \begin{equation}
    \begin{aligned}
        \cL\psi_s&=0\quad\text{in}\ \O,\\
        \psi_s&=1\quad\text{on}\ \partial\O.
    \end{aligned}
    \end{equation} 
\end{example}

    Tables \ref{table:diskp1fem1} and \ref{table:diskp1fem2} contain the convergence rates of the discontinuous Galerkin methods \eqref{eq:discretecon} on uniform meshes and graded meshes (see Figure \ref{fig:lshapedgraded}). As we can see in Table \ref{table:diskp1fem1}, the convergence for the state in $\|\cdot\|_{h}$ and $L_\infty$ norms is approaching $O(h^\frac23)$. This coincides with the theoretical results with $\tau=\alpha=\frac23-\varepsilon$. The convergence of the state in $L_2$ norm is close to $O(h^2)$ and the convergence of the control in $L_2$ norm is approaching $O(h^\frac32)$. These are better than the estimates in Theorem \ref{theorem:main} and consistent with the fact that $\bar{y}\in H^{\frac72-\varepsilon}(\O)$ and $\bar{u}\in H^{\frac32-\varepsilon}(\O)$. We also observe clear improvements of the convergence rates for the state in $\|\cdot\|_{h}$ and $L_\infty$ norms in Table \ref{table:diskp1fem2}. This also coincides with Theorem \ref{theorem:main} with $\tau=1$.

{\scriptsize
\begin{center}
\captionof{table}{Convergence results for Example \ref{ex:lshaped} on uniform meshes}
\vspace{0.3cm}
\begin{tabular}{|c|c|c|c|c|c|c|c|c|}\hline\label{table:diskp1fem1}
$k$&$\|\bar{y}-y_h\|_{L_2(\Omega)}$&Order&$\|\bar{y}-y_h\|_{h}$&Order&$\|\bar{u}-u_h\|_{L_2(\Omega)}$&Order&$\|\bar{y}-y_h\|_{L^\infty(\O)}$&Order\\[1.5ex]
\hline
$1$&3.30e+01&-&2.79e+01&-&6.79e+01&-&7.03e+00&-\\[0.8ex]
\hline
$2$&1.97e+01&0.75&2.39e+01&0.22&3.10e+01&1.13&6.73e+00&0.06\\[0.8ex]
\hline
$3$&7.34e+00&1.42&1.52e+01&0.65&9.37e+00&1.73&4.90e+00&0.46\\[0.8ex]
\hline
$4$&2.02e+00&1.86&8.82e+00&0.79&3.31e+00&1.50&3.08e+00&0.67\\[0.8ex]
\hline
$5$&5.09e-01&1.99&4.66e+00&0.92&1.27e+00&1.39&1.88e+00&0.71\\[0.8ex]
\hline
$6$&1.38e-01&1.88&2.44e+00&0.93&4.99e-01&1.35&1.15e+00&0.71\\[0.8ex]
\hline
$7$&4.76e-02&1.54&1.33e+00&0.87&1.81e-01&1.47&7.15e-01&0.69\\[0.8ex]
\hline
$8$&1.10e-02&2.12&7.64e-01&0.80&5.82e-02&1.63&4.47e-01&0.68\\[0.8ex]
\hline
\end{tabular}

\captionof{table}{Convergence results for Example \ref{ex:lshaped} on graded meshes}
\vspace{0.3cm}
\begin{tabular}{|c|c|c|c|c|c|c|c|c|}\hline\label{table:diskp1fem2}
$k$&$\|\bar{y}-y_h\|_{L_2(\Omega)}$&Order&$\|\bar{y}-y_h\|_{h}$&Order&$\|\bar{u}-u_h\|_{L_2(\Omega)}$&Order&$\|\bar{y}-y_h\|_{L^\infty(\O)}$&Order\\[1.5ex]
\hline
$1$&2.13e+01&-&2.94e+01&-&3.09e+01&-&6.78e+00&-\\[0.8ex]
\hline
$2$&1.14e+01&0.91&2.66e+01&0.15&2.49e+01&0.31&5.07e+00&0.42\\[0.8ex]
\hline
$3$&4.18e+00&1.44&1.78e+01&0.58&4.66e+00&2.42&2.55e+00&0.99\\[0.8ex]
\hline
$4$&1.04e+00&2.01&9.67e+00&0.88&1.77e+00&1.40&9.81e-01&1.38\\[0.8ex]
\hline
$5$&2.52e-01&2.04&4.55e+00&1.09&8.51e-01&1.05&3.15e-01&1.64\\[0.8ex]
\hline
$6$&6.75e-02&1.90&2.03e+00&1.16&3.74e-01&1.19&1.50e-01&1.07\\[0.8ex]
\hline
$7$&3.24e-02&1.06&9.06e-01&1.17&1.66e-01&1.17&5.93e-02&1.34\\[0.8ex]
\hline
$8$&3.99e-03&3.02&4.14e-01&1.13&5.49e-02&1.59&2.80e-02&1.08\\[0.8ex]
\hline
\end{tabular}
\end{center}}

\section{Concluding Remark}\label{sec:conremark}

We propose and analyze discontinuous Galerkin methods to solve an optimal control problem with a general state equation and pointwise state constraints on general polygonal domains. Concrete error estimates are established and numerical results are provided to support the theoretical results. We do not consider convection-dominated case in this paper, hence the constants throughout this paper might depend on $\bz$ and $\gamma$. However, we would like to point out that the potential of our methods is to solve optimal control problems governed by convection-dominated equations with pointwise state constraints. 
There are some previous work \cite{leykekhman2012local,heinkenschloss2010local} concerning the optimal control problems governed by convection-dominated problems without state constraints. As pointed out in \cite{leykekhman2012local}, the weak treatment of the Dirichlet boundary conditions (as we did in \eqref{eq:ddef}) are crucial for optimal control problems governed by convection-dominated equations. 
However, rigorous analysis of the convection-dominated case seems nontrivial. We will investigate this in future work.

\appendix  

\section{Primal-dual active set algorithm}\label{apdix:pdas}

Now we rewrite \eqref{eq:discretecon} in matrix and vector form. Let $\bM_h$ (resp., $\bA_h$) denote the mass (resp., stiffness) matrix represent the bilinear form $(\cdot,\cdot)_\LT$ (resp., $a_h(\cdot,\cdot)$) with respect to the natural discontinuous nodal basis in $V_h$. Assume $\bL_h$ represents $\cL_{h,g}$ and $\mathbf{g}$ represents the boundary integral term in \eqref{eq:ddef}. It follows from \eqref{eq:ddef} that
\begin{equation}
    \mathbf{v}^t\bM_h\bL_h(\mathbf{u})=\mathbf{v}^t\bA_h\mathbf{u}+\mathbf{v}^t\mathbf{g}\quad\forall \mathbf{u},\mathbf{v}\in\mathbb{R}^{n_h},
\end{equation}
where $n_h=\dim V_h$.
This leads to the relation
\begin{equation}\label{eq:bhdef}
    \bL_h(\mathbf{u})=\bM_h^{-1}\bA_h\mathbf{u}+\bM_h^{-1}\mathbf{g}\quad\forall \mathbf{u}\in\mathbb{R}^{n_h},
\end{equation}
and
\begin{equation}\label{eq:bhdeft}
    (\bL_h(\mathbf{u}))^t=\mathbf{u}^t\bA^t_h\bM_h^{-1}+\mathbf{g}^t\bM_h^{-1}\quad\forall \mathbf{u}\in\mathbb{R}^{n_h}.
\end{equation}
Thus \eqref{eq:discretecon} can be rewritten as, by \eqref{eq:bhdef} and \eqref{eq:bhdeft},
\begin{equation}\label{pd}
\begin{aligned}
&\argmin_{\mathbf{y}_h\le\bm{\psi}} \frac{1}{2}(\mathbf{y}_h-\mathbf{y}_d)^t\mathbf{M}_h(\mathbf{y}_h-\mathbf{y}_d)+\frac{\beta}{2}(\bL_h(\mathbf{y}_h))^t\mathbf{M}_h\bL_h(\mathbf{y}_h)\\
&=\argmin_{\mathbf{y}_h\le\bm{\psi}}\frac{1}{2}\mathbf{y}_h^t\mathbf{M}_h\mathbf{y}_h-\mathbf{y}_h^t(\mathbf{M}_h\mathbf{y}_d)\\
&\hspace{1cm}+\frac{\beta}{2}(\mathbf{y}_h^t\bA^t_h\bM_h^{-1}+\mathbf{g}^t\bM_h^{-1})\bM_h(\bM_h^{-1}\bA_h\mathbf{y}_h+\bM_h^{-1}\mathbf{g})\\
&=\argmin_{\mathbf{y}_h\le\bm{\psi}}\frac{1}{2}\mathbf{y}_h^t\left[\beta\mathbf{A}^t_h\mathbf{M}_h^{-1}\mathbf{A}_h+\mathbf{M}_h\right]\mathbf{y}_h-\mathbf{y}_h^t(-\beta\bA^t_h\bM_h^{-1}\mathbf{g}+\mathbf{M}_h\mathbf{y}_d).
\end{aligned}    
\end{equation}

Denote $\mathbf{B}_h=\beta\mathbf{A}^t_h\mathbf{M}_h^{-1}\mathbf{A}_h+\mathbf{M}_h$ and $\tilde{\mathbf{y}}_d=-\beta\bA^t_h\bM_h^{-1}\mathbf{g}+\mathbf{M}_h\mathbf{y}_d$. Let $\mathfrak{n}=\{1,2,\ldots,n_h\}$, the primal-dual active set method for (\ref{pd}) is the following.
\begin{itemize}
\item Given an initial guess $(\mathbf{y}_0, \bm{\lambda}_0)$ where $\bm{\lambda}_0\ge0$, we define
\begin{eqnarray*}
\mathcal{A}_0&=&\{j\in\mathfrak{n}:  \bm{\lambda}_0+c(\mathbf{y}_0(j)-\bm{\psi}(j))>0\},\\
\mathcal{I}_0&=&\{j\in\mathfrak{n}:  \bm{\lambda}_0+c(\mathbf{y}_0(j)-\bm{\psi}(j))\le0\}=\mathfrak{n}\setminus \mathcal{A}_0.
\end{eqnarray*}
\item For $k\ge1$ we recursively solve the system 
\begin{eqnarray}
\label{qd1}\mathbf{B}_h\mathbf{y}_k+\bm{\lambda}_k&=&\tilde{\mathbf{y}}_d,\\
\mathbf{y}_k&=&\bm{\psi}\ \ \ \mbox{on}\ \ \ \mathcal{A}_{k-1},\\
\bm{\lambda}_k&=&0\ \ \ \ \mbox{on}\ \ \ \mathcal{I}_{k-1}.\label{qd3}
\end{eqnarray}
\item Then update the active set and inactive set by
\begin{eqnarray*}
\mathcal{A}_k&=&\{j\in\mathfrak{n}:  \bm{\lambda}_k+c(\mathbf{y}_k(j)-\bm{\psi}(j))>0\},\\
\mathcal{I}_k&=&\{j\in\mathfrak{n}:  \bm{\lambda}_k+c(\mathbf{y}_k(j)-\bm{\psi}(j))\le0\}=\mathfrak{n}\setminus \mathcal{A}_k.
\end{eqnarray*}
\end{itemize}
The unique solution of (\ref{qd1})-(\ref{qd3}) is determined by
\begin{eqnarray}
\nonumber \mathbf{y}_k({\st \mathcal{A}_{k-1}})&=&\bm{\psi}({\st \mathcal{A}_{k-1}}),\\
\nonumber \bm{\lambda}_k({\st\mathcal{I}_{k-1}})&=&0,\\
\mathbf{B}_h({\st \mathcal{I}_{k-1}, \mathcal{I}_{k-1}})\mathbf{y}_k({\st \mathcal{I}_{k-1}})&=&\tilde{\mathbf{y}}_d({\st\mathcal{I}_{k-1}})-\mathbf{B}_h({\st\mathcal{I}_{k-1},\mathcal{A}_{k-1}})\bm{\psi}({\st\mathcal{A}_{k-1}})\label{inactivesys},\\
\nonumber \bm{\lambda}_k({\st\mathcal{A}_{k-1}})&=&\tilde{\mathbf{y}}_{d}({\st \mathcal{A}_{k-1}})-(\mathbf{B}_h\mathbf{y}_{k})({\st\mathcal{A}_{k-1}}).
\end{eqnarray}
\begin{remark}
    Here we follow the MATLAB convention that the vector $\by_k({\st\cA_{k-1}})$ is the subvector of $\by_k$ generated by the components of $\by_k$ corresponding to the index set ${\st\cA_{k-1}}$, the matrix $\mathbf{B}_h({\st\mathcal{I}_{k-1},\mathcal{I}_{k-1}})$ is the submatrix of $\mathbf{B}_h$ generated by the rows and columns of $\mathbf{B}_h$ corresponding to the index set ${\st\mathcal{I}_{k-1}}$, etc.
\end{remark}

\section{Proofs of Lemma \ref{lemma:ah} and Lemma \ref{lemma:rhestimates}}\label{apdix:pfrh}
\begin{proof}[Proof of Lemma \ref{lemma:ah}]
    It is well-known that \cite{arnold2002unified,riviere2008discontinuous,BS}
        \begin{alignat}{3}
            \asiph(w,v)&\le C\trinorm{w}_{h}\trinorm{v}_{h}\quad&&\forall w,v\in \En+V_h,\\
            \asiph(v,v)&\ge C\trinorm{v}^2_{h}\quad&&\forall v\in V_h.\label{eq:asipcoer}
        \end{alignat}
    For the advection-reaction term, we have, for all $w,v\in\En+V_h$,
    \begin{equation*}\label{eq:aarhesti}
        \begin{aligned}
            \aarh(w,v)&=\sum_{T\in\mathcal{T}_h}(\bz\cdot\nabla w+\gamma w, v)_T-\sum_{e\in\cE^i_h\cup\cE^{b,-}_h}(\bn\cdot\bz[w],\{v\})_e\\
            &\lesssim\left(\sum_{T\in\mathcal{T}_h}\|\nabla w\|^2_{L_2(T)}\right)^\frac12\|v\|_\LT+\|w\|_\LT\|v\|_\LT\\
            &\hspace{0.3cm}+\left(\sum_{e\in\cE^i_h\cup\cE^{b,-}_h}\frac{\sigma}{h_e}\|[w]\|^2_{L_2(e)}\right)^\frac12\left(\sum_{e\in\cE^i_h\cup\cE^{b,-}_h}\frac{h_e}{\sigma}\|\{v\}\|^2_{L_2(e)}\right)^\frac12\\
            &\lesssim\trinorm{w}_{h}\trinorm{v}_{h},
        \end{aligned}
    \end{equation*}
    where we use $\bz\in [W^{1,\infty}(\Omega)]^2$, $\gamma\in W^1_{\infty}(\Omega)$, \eqref{eq:traceinq} and \eqref{eq:dgpoin}.
    Furthermore, upon integration by parts, we have, for all $v\in V_h$,
    \begin{align*}
        \aarh(v,v)&=\sum_{T\in\mathcal{T}_h}(\bz\cdot\nabla v+\gamma v, v)_T-\sum_{e\in\cE^i_h\cup\cE^{b,-}_h}(\bn\cdot\bz[v],\{v\})_e\\
        &=\sum_{T\in\mathcal{T}_h}((\gamma-\frac12\nabla\cdot\bz)v,v)_T+\sum_{T\in\mathcal{T}_h}\int_{\partial T}\frac12(\bz\cdot\bn)v^2\ \!ds\\
        &\hspace{0.5cm}-\sum_{e\in\cE^i_h\cup\cE^{b,-}_h}\bz\cdot\bn[v]\{v\}\ \!ds\\
        &=\sum_{T\in\mathcal{T}_h}((\gamma-\frac12\nabla\cdot\bz)v,v)_T+\int_{\partial\O}\frac12|\bz\cdot\bn|v^2\ \!ds.
    \end{align*}
    By the assumption \eqref{eq:advassump}, we immediately have $\aarh(v,v)\ge0$. This finishes the proof.
    \end{proof}

\begin{proof}[Proof of Lemma \ref{lemma:rhestimates}]

It follows from \eqref{eq:ahcont}, \eqref{eq:ahcoer}, \eqref{eq:rh} and \eqref{eq:interpolation} that
\begin{equation}
    \begin{aligned}
    \trinorm{I_hw-R_hw}^2_h&\le Ca_h(I_hw-R_hw,I_hw-R_hw)\\
    &=C a_h(I_hw-w,I_hw-R_hw)\\
    &\le Ch^\tau\|\cL w\|_\LT\trinorm{I_hw-R_hw}_h,
    \end{aligned}
\end{equation}
which implies
\begin{equation}
    \trinorm{I_hw-R_hw}_h\le Ch^\tau\|\cL w\|_\LT.
\end{equation}
Hence we have \eqref{eq:rhh} by triangle inequality. The estimate \eqref{eq:rhl2} is established by a duality argument. Let $\phi\in\Ho$ be defined by
\begin{equation}\label{eq:dualp}
\begin{alignedat}{3}
    -\Delta\phi-\bz\cdot\nabla\phi+(\gamma-\nabla\cdot\bz)\phi&=w-R_hw\quad&&\text{in}\quad\O,\\
    \phi&=0\quad&&\text{on}\quad\partial\O.
\end{alignedat}
\end{equation}
The weak form of the dual problem \eqref{eq:dualp} is to find $\phi\in\Ho$ such that
\begin{equation}
    a(v,\phi)=(v,w-R_hw)\quad\forall v\in\Ho.
\end{equation}
By elliptic regularity \eqref{eq:globalreg}, we have
\begin{equation}\label{eq:dualreg}
    \|\phi\|_{H^{1+\alpha}(\Omega)}\le C_{\Omega}\|w-R_hw\|_{L_2(\Omega)}.
\end{equation}
It follows from \eqref{eq:dualp} that, 
\begin{equation}
    \begin{aligned}
        &\|w-R_hw\|^2_\LT=(-\Delta\phi-\bz\cdot\nabla\phi+(\gamma-\nabla\cdot\bz)\phi,w-R_hw)_\LT\\
        &\quad=(-\Delta\phi,w-R_hw)_\LT+\sum_{T\in\cT_h}(-\bz\cdot\nabla\phi+(\gamma-\nabla\cdot\bz)\phi,w-R_hw)_T.
    \end{aligned}
\end{equation}
By the adjoint consistency of the SIP method, we have
\begin{equation}\label{eq:sipdualconsis}
    (-\Delta\phi,w-R_hw)_\LT=a_h^{sip}(w-R_hw,\phi).
\end{equation}
It follows from integration by parts that
\begin{equation}\label{eq:ardis1}
    \begin{aligned}
        &\sum_{T\in\cT_h}(-\bz\cdot\nabla\phi+(\gamma-\nabla\cdot\bz)\phi,w-R_hw)_T\\
        =&\sum_{T\in\cT_h}(\bz\cdot\nabla(w-R_hw),\phi)_T+(\gamma(w-R_hw),\phi)_T\\
        &\hspace{0.5cm}-\sum_{T\in\cT_h}\int_{\partial T}(\bz\cdot\bn)(w-R_hw)\phi\ \!ds.
    \end{aligned}
\end{equation}
The last term can be rewritten as the following \cite{arnold2002unified,di2011mathematical},
\begin{equation}\label{eq:ardis2}
    \begin{aligned}
        &\sum_{T\in\cT_h}\int_{\partial T}(\bz\cdot\bn)(w-R_hw)\phi\ \!ds\\
        =&\sum_{e\in\cE_h^i}\int_e\bz\cdot\bn [(w-R_hw)\phi]\ \!ds+\sum_{e\in\cE_h^b}\int_e\bz\cdot\bn(w-R_hw)\phi\ \!ds\\
        =&\sum_{e\in\cE_h^i}\int_e\bz\cdot\bn [w-R_hw]\{\phi\}\ \!ds+\sum_{e\in\cE_h^i}\int_e\bz\cdot\bn \{w-R_hw\}[\phi]\ \!ds\\
        &\hspace{0.5cm}+\sum_{e\in\cE_h^b}\int_e\bz\cdot\bn(w-R_hw)\phi\ \!ds.
    \end{aligned}
\end{equation}
It then follows from $[\phi]=0$ on internal edges and $\phi=0$ on $\partial\O$ that
\begin{equation}\label{eq:ardis3}
    \sum_{T\in\cT_h}\int_{\partial T}(\bz\cdot\bn)(w-R_hw)\phi\ \!ds=\sum_{e\in\cE_h^i\cup\cE_h^{b,-}}\int_e\bz\cdot\bn [(w-R_hw)]\{\phi\}\ \!ds.
\end{equation}
According to \eqref{eq:ardis1}-\eqref{eq:ardis3}, we conclude
\begin{equation}\label{eq:ardualconsis}
    \sum_{T\in\cT_h}(-\bz\cdot\nabla\phi+(\gamma-\nabla\cdot\bz)\phi,w-R_hw)_T=a_h^{ar}(w-R_hw,\phi),
\end{equation}
which implies the following together with \eqref{eq:sipdualconsis},
\begin{equation}\label{eq:duall2norm}
    \|w-R_hw\|^2_\LT=a_h(w-R_hw,\phi).
\end{equation}
Therefore, it follows from \eqref{eq:duall2norm}, \eqref{eq:rh}, \eqref{eq:interpolation}, \eqref{eq:ahcont} and \eqref{eq:dualreg} that
\begin{equation}\label{eq:duall2}
    \begin{aligned}
        \|w-R_hw\|^2_\LT&=a_h(w-R_hw,\phi)\\
        &=a_h(w-R_hw,\phi-I_h\phi)\le C\trinorm{w-R_hw}_h\trinorm{\phi-I_h\phi}_h\\
        &\le Ch^\tau\trinorm{w-R_hw}_h\|w-R_hw\|_\LT.
    \end{aligned}
\end{equation}
We then obtain the estimate \eqref{eq:rhl2} combining \eqref{eq:rhh} and \eqref{eq:duall2}.

\end{proof}

\section*{Acknowledgement}
The authors would like to thank Prof. Susanne C. Brenner and Prof. Li-Yeng Sung for the suggestion and discussion regarding this project. The work of the third author was partially supported by the National Science Foundation under grant DMS-2111004.

\bibliographystyle{plain}
\bibliography{references}
 
\end{document}